\documentclass[12pt]{amsart}
\usepackage{a4wide}
\usepackage{amssymb,mathtools}
\usepackage{amsmath}
\usepackage{amsthm}
\usepackage{enumerate}
\usepackage{comment}
\usepackage{graphicx}
\usepackage{hyperref}
\usepackage{amsrefs}

\usepackage{xcolor}

\theoremstyle{plain}

\newtheorem{theorem}{Theorem}[section]

\newtheorem*{claim*}{Claim}
\newtheorem{lemma}[theorem]{Lemma}
\newtheorem{proposition}[theorem]{Proposition}

\newtheorem{corollary}[theorem]{Corollary}

\newtheorem{observation}[theorem]{Observation}

\newcounter{maintheorem}

\newtheorem{mainth}[maintheorem]{Theorem}

\theoremstyle{definition}

\newtheorem{definition}[theorem]{Definition}
\newtheorem{remark}[theorem]{Remark}

\renewcommand{\epsilon}{\varepsilon}

\newcommand{\D}{\mathcal{D}}

\newcommand{\dist}{\mathrm{dist}}

\renewcommand{\leq}{\leqslant}
\renewcommand{\geq}{\geqslant}

\newcommand{\N}{\mathbb{N}}

\newcommand{\conv}{\mathrm{conv}}

\renewcommand{\phi}{\varphi}

\hypersetup{
    pdftitle={},
    pdfauthor={},
    pdfmenubar=false,
    pdffitwindow=true,
    pdfstartview=FitH,
    colorlinks=true,
    linkcolor=black,
    citecolor=teal,
    urlcolor=cyan
}

\newcommand{\nn}[1]{{\left\vert\kern-0.25ex\left\vert\kern-0.25ex\left\vert #1 
\right\vert\kern-0.25ex\right\vert\kern-0.25ex\right\vert}}

\newcommand{\seminorm}[1]{\left\lvert\hspace{-1 pt}\left\lvert\hspace{-1 pt}\left\lvert {#1}\right\lvert\hspace{-1 pt}\right\lvert\hspace{-1 pt}\right\lvert}

\author[C.A.~De~Bernardi]{Carlo Alberto De Bernardi}
\address{Dipartimento di Matematica per le Scienze economiche, finanziarie ed attuariali, Universit\`a Cattolica del Sacro Cuore, 20123 Milano, Italy}
\email{carloalberto.debernardi@unicatt.it} \email{carloalberto.debernardi@gmail.com}

\author[J.~Somaglia]{Jacopo Somaglia}
\address{Politecnico di Milano, Dipartimento di Matematica, Piazza Leonardo da Vinci 32, 20133 Milano, Italy.}
\email{jacopo.somaglia@polimi.it}

\thanks{
The research of the first author was partially supported by the INdAM -GNAMPA Project,  CUP E53C23001670001, and by the MICINN project PID2020-112491GB-I00 (Spain). The research of the second author was partially supported by the INdAM -GNAMPA Project,  CUP E53C23001670001.
}

\title[Almost LUR points]{Some remarks on almost locally uniformly rotund points}

\begin{document}
	
	\begin{abstract} 
	We study the relations between different notions of almost locally uniformly rotund points that appear in literature. We show that every non-reflexive Banach space admits an equivalent norm having a point in the corresponding unit sphere  which is not almost locally uniformly rotund, and
which
    is strongly exposed by all its supporting functionals. This result is in contrast with a characterization due to P.~Bandyopadhyay, D.~Huang, and B.-L.~Lin from 2004. We also show that such a characterization remains true in reflexive Banach spaces.
	\end{abstract}
\maketitle

\section{Introduction}

The aim of the present paper is to study some relations between different  rotundity properties of a given point $x$ belonging to the unit sphere $S_X$ of a Banach space $X$.   Several  notions of rotundity have been introduced and widely studied in the literature. The most common  are the notions of  {\em extreme}, {\em rotund}, and  {\em local uniformly rotund} point (LUR point in short). It is well-known and easy-to-prove that if $x$ is an LUR point then  {\em $x$ is strongly exposed by all its supporting functionals}. In the sequel, points satisfying this last condition are called {\em nicely strongly exposed} (NSE point in short), also known as strongly convex, see \cite{ShuZali}. A notion closely related to that of nicely strongly exposed points has also been studied in the context of optimization of convex functions under the denominations \textit{small diameter property} in \cite{DEBEMIMOSOM} and \textit{strongly adequate functions} in \cite{VolleZali}. 

The following definition has been  introduced and studied in \cite{BaHuLiTr00}, as a generalization of locally uniform rotundity:
 $x$ is an \textit{almost locally uniformly rotund} point (aLUR point in short) if,  for every  pair of sequences $\{x_n\}_{n}\subset S_X$ and $\{x_m^*\}_{m}\subset S_{X^*}$ such that 
\begin{equation*}
		\textstyle	\lim_m\left(\lim_n x_m^*\left(\frac{x_n+x}{2}\right)\right)=1,
\end{equation*} 
then $\{x_n\}_{n}$ converges to $x$. In \cite{BaHuLi04}*{Corollary 4.6} the authors provide several characterizations of aLUR points and, in particular, claim that $x\in S_X$ is an aLUR point if and only if it is an NSE point. Since then, this characterization has been quoted in several papers (see, e.g., \cites{BL, DEBESOMALUR, BLLN, BG}). We refer to \cite{Z25} for a detailed list of papers dealing with the notion of aLUR. The fact that every aLUR point is an NSE point is an easy exercise (for a proof see Observation~\ref{obs: aLUR implies NSE} below). Unfortunately, the proof of the other implication, provided in \cite{BaHuLi04},  contains a gap (see Remark \ref{r: BaHuLigap} below). 
The main aim of our paper is to prove the following characterization of reflexive spaces
\begin{mainth}\label{th: thm A}
 A Banach space $X$ is reflexive if and only if for every equivalent norm $\|\cdot\|$ on $X$ the set of all aLUR points of $S_{(X,\|\cdot\|)}$ coincides with the set of all NSE points of $S_{(X,\|\cdot\|)}$.
\end{mainth} 

\noindent In particular, we show that the equivalence between aLUR and NSE does not hold in general.

Let us briefly describe the structure of the paper. In Section 2, after some notation and preliminaries, in addition to  
aLUR and NSE properties, we introduce other closely related definitions and we study the most immediate relations between them. In Section 3, we state and prove the main results of the paper.
Theorem~\ref{t: nonreflexive} shows  that the equivalence between aLUR and NSE fails, in the following strong sense:  
{\em every non-reflexive Banach space admits an equivalent norm such that the corresponding unit sphere contain an NSE point which is not aLUR}.
On the other hand, Theorem~\ref{th: equiv aLUR} shows that the equivalence between  aLUR and NSE holds whenever $X$ is a reflexive Banach space. Combining these two results, we obtain a characterization of reflexive Banach spaces (see Corollary \ref{c: characterizationreflexive}). Moreover, in the spirit of \cites{BaHuLi04, BaHuLiTr00}, in Theorem~\ref{t: caratNSE} we give a characterization of property NSE in terms of double limit. It turns out that NSE is equivalent to a property that is similar to aLUR but in which, roughly speaking, instead of taking iterated limits, we consider convergence of the double limit in the  Pringsheim's sense (i.e., letting both indexes tend to $\infty$, independently of each other).

\section{Basic notions}
We follow the notation and terminology introduced in \cite{DEBESOMALUR}. Throughout this  paper, all Banach spaces are real and infinite-dimensional. Let $X$ be a Banach space, by $X^*$ we denote the dual space of $X$. By $B_X$ and $S_X$ we denote the closed unit ball and the unit sphere of $X$, respectively.
	Moreover, in situations when more than one
	norm on $X$ is considered, we denote by $B_{(X,\|\cdot\|)}$ and $S_{(X,\|\cdot\|)}$ the closed unit ball and the closed unit sphere with respect to the norm ${\|\cdot\|}$, respectively. By $\|\cdot\|^*$, $\|\cdot\|^{**}$, $\|\cdot\|^{***}$ we denote the dual, bidual, third dual  norm of $\|\cdot\|$, respectively.  For $x,y\in X$, $[x,y]$ denotes the closed segment in $X$ with
endpoints $x$ and $y$.
    
 A \emph{biorthogonal system} in a separable Banach space $X$ is a system $(e_n, f_n)_{n}\subset X\times X^*$, such that $f_n(e_m)=\delta_{n,m}$ ($n,m\in \N$). A biorthogonal system is \emph{fundamental} if ${\rm span}\{e_n\}_{n}$ is dense in $X$; it is \emph{total} when ${\rm span}\{f_n\} _{n}$ is $w^*$-dense in $X^*$. A \emph{Markushevich basis} (M-basis) is a fundamental and total biorthogonal system. We refer to \cites{HMVZ,HRST,RS23} and references therein for more information on M-bases.

Let us recall that the {\em duality map} $\D_X: S_X\to2^{S_{X^*}}$ is the function defined, for each $x\in S_X$, by
$$\D_X(x):=\{x^*\in S_{X^*}\colon\, x^*(x)=1\}.$$

\begin{definition}\label{d: alurs}
Let $x\in S_X$. We say that:
	\begin{enumerate}
    \item $x$ is a {\em rotund} point of $B_X$ if for $y\in S_X$, such that $\|y+x\|=2$, we have $x=y$; 
    \item $x$ is  {\em strongly exposed} by $x^*\in S_{X^*}$ if 
 $x_n\to x$ for all sequences $\{x_n\}_n\subset B_X$ such that $\lim_{n\to \infty}x^*(x_n)=1$;
 \item\label{alur3} $x$ is a \textit{nicely strongly exposed point} of $S_X$ (or satisfies property NSE) if $x$ is strongly exposed by $x^*$, whenever  $x^*\in\D_X(x)$;
		\item\label{alur1} $x$ is an \textit{almost locally uniformly rotund point} of $S_X$ (or satisfies property aLUR) if,  for every  pair of sequences $\{x_n\}_{n}\subset S_X$ and $\{x_m^*\}_{m}\subset S_{X^*}$ such that \begin{equation}\label{eq: doppiolimite}
			\lim_m\left(\lim_n x_m^*\left(\frac{x_n+x}{2}\right)\right)=1,
		\end{equation} we have that $\{x_n\}_{n}$ converges to $x$;

				\item\label{alur1'} $x$ satisfies property aLUR$'$ if,  for every  pair of sequences $\{x_n\}_{n}\subset S_X$ and $\{x_m^*\}_{m}\subset S_{X^*}$ such that 
    \begin{equation}\label{eq: doppiolimiteliminf}
			\lim_m\left(\liminf_n \,x_m^*\left(\frac{x_n+x}{2}\right)\right)=1,
		\end{equation} 
        we have that $\{x_n\}_{n}$ converges to $x$.
\end{enumerate}
\end{definition}
	
	\begin{remark}\label{remark: aLUR}	    The ``double limit'' in \eqref{eq: doppiolimite} has to be intended in the  sense indicated by the brackets: 
		\begin{itemize}
			\item for every $m\in\N$, the limit $\alpha_m:=\lim_n x_m^*((x_n+x)/2)$ exists;
			\item $\lim_m \alpha_m=1$.
		\end{itemize}
    The formula in \eqref{eq: doppiolimiteliminf} should be read analogously.
 	\end{remark}

\noindent The following proposition shows that properties aLUR and aLUR$'$ are indeed equivalent.
\begin{proposition}\label{p: liminf}
    		Let $x\in S_X$. The following assertions are equivalent:
		\begin{enumerate}
			\item\label{P1} $x$ satisfies property aLUR;
			\item\label{P4} $x$ satisfies property aLUR$'$.
		\end{enumerate}

\end{proposition}

\begin{proof}
The implication $\eqref{P4} \Rightarrow \eqref{P1}$ is trivial. Suppose that $\eqref{P1}$ holds and suppose on the contrary that $x\in S_X$ does not satisfy property aLUR$'$. This means that there exist $\varepsilon> 0$ and two sequences $\{x_n\}_{n} \subset S_X$ and $\{x^*_m\}_{m}\subset S_{X^*}$ such that 
    \begin{equation*}
\textstyle 			\lim_m\left(\liminf_n \,x_m^*\left(\frac{x_n+x}{2}\right)\right)=1
		\end{equation*} 
holds, and $\|x_n-x\|>\varepsilon$ for every $n\in\N$.  By  definition of $\liminf$, there exists a subsequence $\{x_n^1\}_{n}$ of $\{x_n\}_{n}$ such that $\lim_{n} x_1^*(\frac{x_n^1+x}{2})=\liminf_{n} x^*_1(\frac{x_n+x}{2})=:\alpha_1$. By iterating the argument, we can inductively define sequences $\{x_n^m\}_{n}$ ($m\in\N$) such that, for every $m$, we have:
\begin{itemize}
    \item $\{x_n^{m+1}\}_{n}$ is a  subsequence of $\{x_n^m\}_{n}$;
    \item $\lim_{n} x_{m+1}^*\left(\frac{x_n^{m+1}+x}{2}\right)=\alpha_{m+1}:= \liminf_{n} x^*_{m+1}\left(\frac{x_n^m+x}{2}\right)$.
\end{itemize}
Now, let us consider the  subsequence $\{y_{n}\}_{n}$ of $\{x_n\}_{n}$ defined by $y_n=x^n_n$ ($n\in\N$), and observe that, since \[\textstyle\liminf_{n} x^*_{m+1}\left(\frac{x_n^m+x}{2}\right)\geq \liminf_{n} x^*_{m+1}\left(\frac{x_n+x}{2}\right)\qquad (m\in\N),\]
we have \begin{equation*}
    \textstyle  
    1\geq \lim_m\left(\lim_n \,x_m^*\left(\frac{y_{n}+x}{2}\right)\right)=\lim_m \alpha_m \geq \lim_m\left(\liminf_n \,x_m^*\left(\frac{x_n+x}{2}\right)\right) =1.
\end{equation*}
Therefore, by  assumption \eqref{P1}, the sequence $\{y_{n}\}_{n}$ must converge to $x$, which is in contradiction with $\|x_n-x\|>\varepsilon$ for every $n\in\N$. Thus $\eqref{P1} \Rightarrow \eqref{P4}$.
\end{proof}

\begin{remark} In \cite{Z25} the author suggested to use, as the definition of almost locally uniformly rotund point, the one we denote by aLUR$'$. The main reason of this suggestion is the fact that the proof of \cite{BaHuLiTr00}*{Theorem 6} contains a gap (i.e. the existence of a limit). This gap can be fixed by using aLUR$'$ instead of aLUR (see \cite{Z25}*{Theorem E}). In light of Proposition~\ref{p: liminf}, it turns out that the two definitions  aLUR and aLUR$'$ are equivalent. Therefore, combining \cite{Z25}*{Theorem E} with Proposition~\ref{p: liminf} we get that \cite{BaHuLiTr00}*{Theorem 6} is correct without using a different definition. 
\end{remark}

We conclude this section by recalling the definition and a geometric characterization of the weak counterpart of aLUR.

\begin{definition}\label{walur}
    Let $x\in S_X$. We say that $x$ is a \textit{weakly almost locally uniformly rotund} $w$-aLUR point of $S_X$ if,  for every  pair of sequences $\{x_n\}_{n}\subset S_X$ and $\{x_m^*\}_{m}\subset S_{X^*}$ such that \begin{equation*}
			\lim_m\left(\lim_n x_m^*\left(\frac{x_n+x}{2}\right)\right)=1,
		\end{equation*} we have that $\{x_n\}_{n}$ converges weakly to $x$.
\end{definition}
\noindent
The following result, due to Bandyopadhyay, Huang, Lin and Troyanski is contained in \cite{BaHuLiTr00}, will be used in Theorem \ref{t: nonreflexive} for proving that a certain point of the unit sphere is not $w$-aLUR. 

\begin{theorem}[\cite{BaHuLiTr00}*{Corollary 8}]\label{th: char w-aLUR}
    Let $X$ be a Banach space. For $x\in S_X$, the following are equivalent
    \begin{enumerate}
        \item $x$ is a rotund point of $B_{X^{**}}$;
        \item $x$ is a $w$-aLUR point of $S_X$.
    \end{enumerate}
\end{theorem}

\begin{remark}
Clearly, if a point is aLUR, it is $w$-aLUR. The converse is not true in general. Indeed, in \cite{DEBEPRESOMMLUR}*{Section 2.1}, it is provided a norm which is WLUR, therefore each point of the unit sphere is $w$-aLUR, but it is not MLUR, thus there exists a point in the unit sphere which is not aLUR (see also \cite{D15}).
\end{remark}

\section{Main results}

This section is devoted to the study of relations between  the different notions introduced in Definition \ref{d: alurs}. Let us start with the following easy-to-prove observation, asserting that  property aLUR implies property NSE. For the sake of completeness we include a proof of it.

\begin{observation}\label{obs: aLUR implies NSE} Let $x\in S_X$. If $x$ satisfies property aLUR then $x$ satisfies property NSE.    
\end{observation}

\begin{proof} Assume that $x$ satisfies aLUR. Let $x^*\in\D_X(x)$ and $\{x_n\}_{n}\subset S_X$ be such that $x^*(x_n)\to 1$. Then, if we define $x_m^*=x^*$ ($m\in\N$), we have that \eqref{eq: doppiolimite} is satisfied and hence $x_n\to x$ (since $x$ satisfies property aLUR). We have proved that $x$ is strongly exposed by $x^*$ and hence, by the arbitrariness of $x^*\in\D_X(x)$,  that  $x$ satisfies NSE.     
\end{proof}

	\begin{theorem}\label{th: equiv aLUR}
    Let $X$ be a reflexive Banach space. If $x\in S_X$ satisfies NSE, then $x$ satisfies aLUR.
	\end{theorem}

\begin{proof} Let $x$ be an NSE point of $S_X$, and suppose on the contrary that $x$ does not satisfy aLUR. Then, passing to suitable subsequences if necessary, it is easy to see that there exist $\epsilon>0$  and sequences  $\{x_n\}_{n}\subset S_X$, $\{x_m^*\}_{m}\subset S_{X^*}$ such that: \begin{itemize}	\item \eqref{eq: doppiolimite} holds, that is: 	\begin{itemize} 		\item for $m\in\N$, the limit $\alpha_m:=\lim_n(x_m^*((x_n+x)/2))$ exists; 		\item $\lim_m \alpha_m=1$; 	\end{itemize} 	\item $\|x-x_n\|>\epsilon$, for every $n\in\N$; 	\item there exists $l:=\lim_m x_m^*(x)$.  \end{itemize} Let $x^*\in B_{X^*}$ be a $w^*$-cluster point of the sequence $\{x_m^*\}_{m}\subset S_{X^*}$.  Then, from \eqref{eq: doppiolimite}, we have that $l=1$ and hence that $x^*\in\D_X(x)$. Since $\|x-x_n\|>\epsilon$ and since $x$ satisfies NSE, there exists $\delta\in(0,1)$ such that $x^*(x_{n})\leq\delta$, whenever $n\in\N$. Now,  let $y\in B_X$ be a $w$-cluster point of the sequence $\{x_{n}\}_{n}$ and observe that by our hypothesis we have $\alpha_m:=x_m^*((y+x)/2)$, whenever $m\in\N$. Hence, $x^*((y+x)/2)=1$ and $x^*(y)=1$. A contradiction since $y$ is a $w$-cluster point of the sequence $\{x_{n}\}_{n}$ and $x^*(x_{n})\leq\delta$, whenever $n\in\N$.
\end{proof}

Then, in the spirit of \cites{BaHuLi04, BaHuLiTr00}, we provide  a characterization of property NSE in terms of double limit. To do this, we introduce  a slight variant of aLUR property, in which, roughly speaking, instead of taking iterated limits, we consider convergence of the double limit in the  Pringsheim's sense. The idea of the proof is analogous  to the one of Theorem~\ref{th: equiv aLUR}.

\begin{theorem}\label{t: caratNSE}
    		A point $x\in S_X$ satisfies property NSE if and only if 
		for every  pair of sequences $\{x_n\}_{n}\subset S_X$ and $\{x_m^*\}_{m}\subset S_{X^*}$ such that \begin{equation}\label{eq: doppiolimiteunif}
		\lim_{m,n}x_m^*\left(\frac{x_n+x}{2}\right)=1,
	\end{equation} we have that $\{x_n\}_{n}$ converges to $x$, where the ``double limit'' in \eqref{eq: doppiolimiteunif} is to be intended in the following sense: for every $\epsilon>0$ there exists $n_0\in\N$ such that if  $n,m\geq n_0$ then $\left|x_m^*\left(\frac{x_n+x}{2}\right)-1\right|<\epsilon$ (equivalently, $x_m^*\left(\frac{x_n+x}{2}\right)>1-\epsilon$).	 
\end{theorem}

\begin{proof}
Let us prove the sufficiency part of the theorem. 
Let $x^*\in\D_X(x)$ and $\{x_n\}_{n}\subset S_X$ be such that $x^*(x_n)\to 1$. 
Then, if we define $x_m^*=x^*$ ($m\in\N$), \eqref{eq: doppiolimiteunif} is satisfied and, by our assumption, $x_n\to x$. We have proved that $x$ is strongly exposed by $x^*$ and hence, by the arbitrariness of $x^*\in\D_X(x)$,   that $x$ satisfies property NSE.

For the other implication, assume that $x$ satisfies property NSE and suppose on the contrary (passing to suitable subsequences if necessary) that there exist $\theta>0$  and sequences  $\{x_n\}_{n}\subset S_X$, $\{x_m^*\}_{m}\subset S_{X^*}$ such that:
\begin{itemize}
	\item \eqref{eq: doppiolimiteunif} holds, that is: for every $\epsilon>0$ there exists $n_0\in\N$ such that if  $n,m\geq n_0$ then $x_m^*\left(\frac{x_n+x}{2}\right)>1-\epsilon$;
	\item $\|x-x_n\|>\theta$, whenever $n\in\N$;
	\item there exists $l:=\lim_m x_m^*(x)$. 
\end{itemize} Let $x^*\in B_{X^*}$ be a $w^*$-cluster point of the sequence $\{x_m^*\}_{m}\subset S_{X^*}$.  Then, from \eqref{eq: doppiolimiteunif}, we have that $l=1$ and hence that $x^*\in\D_X(x)$. Since $x$ satisfies property NSE and $\|x-x_n\|>\theta$ for every $n\in\N$, there exists $\delta\in(0,1)$ such that $x^*(x_{n})\leq\delta$, whenever $n\in\N$. Now,  let $x^{**}\in B_{X^{**}}$ be a $w^*$-cluster point of the sequence $\{x_{n}\}_{n}$ and take any $\epsilon>0$, by our hypothesis there exists $n_0\in\N$ such that if  $n,m\geq n_0$ then $x_m^*\left(\frac{x_{n}+x}{2}\right)>1-\epsilon$. In particular, 
 if  $n\geq n_0$ then $x^*\left(\frac{x_{n}+x}{2}\right)>1-\epsilon$
 and hence $x^*\left(\frac{x^{**}+x}{2}\right)>1-\epsilon$. %\textcolor{blue}{By the arbitrariness of $\epsilon>0$ we have that $x^*((x^{**}+x)/2)=1$ and hence that $x^{**}=x$ (since strongly exposed points of $S_X$ are strongly exposed points of the second dual by the same functional, see, e.g., \cite{FHHMZ}*{Exercise~7.74}). Since $x^{**}$ is a $w^*$-cluster point of the sequence $\{x_{n}\}_{n}$ and $x^*(x_{n})\leq\delta$, whenever $n\in\N$, we have $$1=x^*(x)=x^{**}(x^*)\leq\delta<1,$$  a contradiction and our conclusion holds.}

 By the arbitrariness of $\epsilon>0$ we have 
 that $x^*((x^{**}+x)/2)=1$, which implies $x^{**}(x^*)=1$. We get a contradiction, since $x^{**}$ is a $w^*$-cluster point of the sequence $\{x_{n}\}_{n}$ and $x^*(x_{n})\leq\delta$, whenever $n\in\N$.
\end{proof}

By using a standard technique about extension of norms (see \cite{DGZ}*{Lemma~8.1 in \S II}), the following lemma shows that, if a given equivalent norm on a  subspace $Y$ of $X$ is  extended in a suitable way to the whole  $X$, then NSE points of $S_Y$ are automatically NSE points  of $S_X$. For the sake of completeness we provide a sketch of the proof.

\begin{lemma}\label{lemma: NSE superspace}
Let $Y$ be a closed subspace of a Banach space $X$ and let $\|\cdot\|$ be an equivalent norm on $Y$. Then $\|\cdot\|$ can be extended to an equivalent norm $\seminorm{\cdot}$ on $X$ such that every  NSE point of $S_{(Y,\|\cdot\|)}$ is an NSE point of $S_{(X,\seminorm\cdot)}$.
\end{lemma}

\begin{proof} Without any loss of generality, we can suppose that  $\|\cdot\|$ is defined on the whole space $X$ (see \cite{DGZ}*{Lemma~8.1 in \S II} or \cite{FHHMZ}*{Proposition~2.14}). Let us define 
    \begin{equation*}
        \seminorm{x}^2=\|x\|^2 + \dist^2_{\|\cdot\|}(x,Y),\qquad\qquad x\in X.
    \end{equation*}
    Notice that $ \seminorm{y}=\|y\|$ for every $y\in Y$. Assume that  $y_0$ is an NSE point of $S_{(Y,\|\cdot\|)}$, and let us prove that $y_0$ is an NSE point of $S_{(X,\seminorm{\cdot})}$. To do that, let $x^*\in \D_{(X,\seminorm{\cdot})}(y_0)$  and  $\{x_n\}_n\subset S_{(X,\seminorm{\cdot})}$ be such that $x^*(x_n)\to 1$. Since
    \begin{equation*}
        2\geq \seminorm{x_n + y_0}\geq x^*(x_n+y_0) = x^*(x_n) + 1 \to 2,
    \end{equation*}
    it holds $\seminorm{x_n + y_0}\to 2$. Therefore
\begin{equation*}    2\seminorm{x_n}^2+2\seminorm{y_0}^2-\seminorm{x_n + y_0}^2\to 0,
\end{equation*}
    which implies $\dist_{\|\cdot\|}(x_n,Y)\to 0$ (see \cite{DGZ}*{Fact~2.3 in \S II}). Now, let $y_n \in Y$ $(n\in\N)$ be such that $\seminorm{y_n}=1$ and $\seminorm{x_n - y_n}\to 0$. Notice that
    \begin{equation*}
        x^*(y_n)=x^*(y_n-x_n)+x^*(x_n)\to 1.
    \end{equation*}
    Since $x^*|_Y\in \D_{(Y,\|\cdot\|)}(y_0)$ and $y_0$ is an NSE point of $S_{(Y,\|\cdot\|)}$, we have $\|y_n - y_0\|=\seminorm{y_n-y_0}\to 0$. Which clearly implies $x_n\to y_0$ in $(X,\seminorm{\cdot})$. By the arbitrariness of $x^*\in \D_{(X,\seminorm{\cdot})}(y_0)$, we have proved  that $y_0$ is an NSE point of $S_{(X,\seminorm{\cdot})}$.
\end{proof}

In the proof of our next theorem, we shall need the following easy observation. We provide a proof based on Theorem~\ref{th: char w-aLUR}. Alternatively, the proposition can be proved by using the very definition of $w$-aLUR and the Hahn-Banach theorem.

\begin{observation}\label{obs: not aLUR superspace} Let $Y$ be a closed subspace of a Banach space $X$. If $y_0\in S_Y$ is not a $w$-aLUR point of $S_Y$ then $y_0$ is not a $w$-aLUR point of $S_X$.
\end{observation}

\begin{proof}
    Assume that  $y_0\in S_Y$ is not a $w$-aLUR point of $S_Y$. By Theorem~\ref{th: char w-aLUR}, $y_0$ is not a rotund point of $B_{Y^{**}}$, that is, there exists $y^{**}\in S_{Y^{**}}\setminus\{y_0\}$ such that the segment $[y^{**},y_0]$ is contained in $S_{Y^{**}}$. Since $Y^{**}\subset X^{**}$ and $S_{X^{**}}\cap Y^{**}=S_{Y^{**}}$, the segment $[y^{**},y_0]$ is contained in $S_{X^{**}}$. By Theorem~\ref{th: char w-aLUR}, $y_0$ is not a $w$-aLUR point of $S_X$.
\end{proof}

We are now ready to show that in general property NSE does not imply property aLUR. More precisely, our next result shows that if $X$ is a non-reflexive Banach space, then it admits a renorming which satisfies NSE but not $w$-aLUR (and hence not aLUR) at a certain point $x\in S_X$. The construction is in part inspired by \cite{DEBESOMALUR}.

\begin{theorem}\label{t: nonreflexive}
    Let $(X,\|\cdot\|)$ be a non-reflexive Banach space. Then $X$ admits a renorming $\seminorm{\cdot}$ such that there exists $x\in S_{(X,\seminorm{\cdot})}$ which is NSE but not $w$-aLUR (and hence not aLUR).
\end{theorem}

\begin{proof} Since every non-reflexive Banach space admits a non-reflexive separable subspace, in light of Lemma~\ref{lemma: NSE superspace} and Observation~\ref{obs: not aLUR superspace}, it is sufficient to prove the theorem in the case in which $X$  is separable. So, in the sequel of the proof, we suppose that $X$ is a non-reflexive separable Banach space.

We first observe that, since $X$ is non-reflexive, there exist $x^{**}_1\in S_{(X^{**},\|\cdot\|^{**})}$ and $x^{***}_1\in S_{(X^{***},\|\cdot\|^{***})}$ such that
\begin{enumerate}
    \item $x^{***}_1(x_1^{**})=1$;
    \item $0<\sup_{x\in B_{(X,\|\cdot\|)}}|x^{***}_1(x)|<\frac{1}{3}$.
\end{enumerate}
Indeed, we recall that $X^{***}=X^{\perp}\oplus X^*$ (see, e.g., \cite{FHHMZ}*{Exercise~4.7}). Let $x^{***}\in X^\perp$ be such that $\|x^{***}\|^{***}=1$. By the Bishop-Phelps theorem there exists $x^{***}_1 \in S_{(X^{***},\|\cdot\|^{***})}$ such that $\|x^{***}-x^{***}_1\|^{***}<\frac{1}{3}$, $x^{***}_1$ attains its norm on $B_{(X^{**},\|\cdot\|^{**})}$ at a certain point $x^{**}$, and $x^{***}_1\notin X^\perp$. By defining $x^{**}_1:= x^{**}$ we get (i). Moreover, for $x\in B_{(X,\|\cdot\|)}$, we have
\begin{equation*}
    |x^{***}_1(x)|=|x^{***}(x)-x^{***}_1(x)|\leq \|x\|\|x^{***}-x^{***}_1\|^{***}<\frac{1}{3}.
\end{equation*}
This proves (ii). In order to simplify our notation, we define $x^{***}_1|_{X}=x_1^*$. Notice that, since $x_1^{***}\notin X^{\perp}$, $x^*_1$ is different from the zero functional. Let $x_1\in X$ be such that $x^*_1(x_1)=1$, we clearly have that $x_1\notin B_{(X,\|\cdot\|)}$.  Then, by \cite{FHHMZ}*{Theorem~4.60}, $X$ has an M-basis $(e_n,f_n)_{n}$ such that $e_1=x_1$, $f_1=x_1^*$, and $\|e_n\|=1$, for every $n\geq 2$.  We consider the linear operator $T\colon (\ell_2,\|\cdot\|_2) \to (X,\|\cdot\|) $ defined by 
\[  T(\alpha)=\sum_{n=1}^{\infty} \frac{\alpha_n}{n^2} e_n,\qquad\qquad \alpha=(\alpha_n)_n\in\ell_2.\]
  We notice that the operator $T$ is well-defined, bounded, linear, one-to-one, and the range $Y:=T(\ell_2)$ contains $\{e_n\}_{n}$, therefore $Y$ is dense in $X$. %By the injectivity of the operator $T$, we can consider the subspace $Y$ endowed with the norm $\|\cdot\|_{\theta}$ defined by $\|y\|_{\theta}:=\|T^{-1}(y)\|_2$, for any $y\in Y$. In this way, we obtain that $T$ is an isometric isomorphism between $(\ell_2,\|\cdot\|_2)$ and $(Y,\|\cdot\|_{\theta})$. 
 We set $B:=T(B_{\ell_2})$. 
  %In other words, we have that
%\begin{equation*}
 %   B=\{y\in Y\colon \|y\|_{\theta}\leq 1\}.
%\end{equation*}
%\textcolor{blue}{By a standard argument, it is not difficult  to see that the convex subset $B$ is compact in $(X,\|\cdot\|)$ (see \cite{DEBESOMALUR}*{Section 3} for more details).} 
Since $T$ is $w$-$w$-continuous and $B_{\ell_2}$ is $w$-compact, $B$ is $w$-compact.  Hence, we have that the set  $$D=\conv\bigl(B_{{(X,\|\cdot\|)}}\cup B\bigr)$$
	is closed in ${{X}}$.  
	Then by our definition and by symmetry,  $D$ is the closed unit ball of an equivalent norm $\seminorm{\cdot}$ on ${X}$. 
    \smallskip 
    
    \noindent\textit{Claim:} $B_{(X^{**},\seminorm{\cdot}^{**})}=\conv(B_{(X^{**},\|\cdot\|^{**})}\cup B).$\\
By applying Goldstine's Theorem, we have 
\begin{equation*}
\begin{split}
B_{(X^{**},\seminorm{\cdot}^{**})}&=\overline{D}^{w^*}=\overline{\conv}^{w^*}(B\cup B_{(X,\|\cdot\|)})\subset\overline{\conv}^{w^*}(\overline{B}^{w^*} \cup \overline{B_{(X,\|\cdot\|)}}^{w^*} )\\
&=\conv(B\cup \overline{B_{(X,\|\cdot\|)}}^{w^*})=\conv(B\cup B_{(X^{**},\|\cdot\|^{**})}).
\end{split}
\end{equation*}
On the other hand, we have $B\subset B_{(X,\seminorm{\cdot})}\subset B_{(X^{**},\seminorm{\cdot}^{**})}$, and 
\begin{equation*}
B_{(X^{**},\|\cdot\|^{**})}=\overline{B_{(X,\|\cdot\|)}}^{w^*}\subset \overline{\conv}^{w^*}(B\cup B_{(X,\|\cdot\|)})=\overline{D}^{w^*}= B_{(X^{**},\seminorm{\cdot}^{**})}.
\end{equation*}
Therefore, $\conv{(B_{(X^{**},\|\cdot\|^{**})}\cup B)}\subset B_{(X^{**},\seminorm{\cdot}^{**})} $, which proves the claim.
\smallskip

We are  going to show that $\seminorm{x^{***}_1}^{***}=1$. Indeed, $\sup_{x^{**}\in B_{(X^{**},\|\cdot\|^{**})}}|x^{***}_1(x^{**})|=1$. Thus, in light of the claim, it is enough to show that $\sup_{x\in B}|x_1^{***}(x)|\leq 1$. In order to do that, let $x\in B$. By the definition of the operator $T$, there exists $\alpha \in B_{\ell_2}$ such that $T(\alpha)=x$. Hence, we get
\begin{equation*}
  |x_1^{***}(x)|=|x^*_1(x)|=|x_1^{*}(T\alpha)|=\left|f_1\left(\sum_{n=1}^{\infty} \frac{\alpha_n}{n^2} e_n\right)\right|\leq\sum_{n=1}^{\infty} \frac{|\alpha_n|}{n^2} |f_1(e_n)|\leq|\alpha_1|\leq 1.
\end{equation*}
Since $x^{***}_1(x_1^{**})=1$, $x^{***}_1(e_1)=1$, and $\seminorm{x_1^{***}}^{***}=1$, we have that the segment $[ e_1, x_1^{**}]$ is contained in  $S_{(X^{**},\seminorm{\cdot}^{**})}$. Hence, the point $e_1\in X$ is not a rotund point of $B_{(X^{**},\seminorm{\cdot}^{**})}$, which shows, by Theorem~\ref{th: char w-aLUR}, that $e_1$ is not a \textit{w}-aLUR point (and hence in particular that $e_1$ is not an aLUR point) of $S_{(X,\seminorm{\cdot})}$. 

It remains to show that $e_1$ is an NSE point of $S_{(X,\seminorm{\cdot})}$. %\textcolor{blue}{By proceeding as in Case 2 of \cite{DEBESOMALUR}*{Proposition 3.6}, it is possible to prove that $e_1$ is a smooth point of the ball $B_{(X,\seminorm{\cdot})}$.} 
In order to do that, we first show that $e_1$ is a smooth point of the ball $B_{(X,\seminorm{\cdot})}$. Indeed, suppose on the contrary that there exist two distinct functionals $h_1,h_2\in S_{(X^*,\seminorm{\cdot}^*)}$ such that $h_1(e_1)=h_2(e_1)=1.$
Put $u^*_i=T^*(h_i).$ Since $T$ has dense range, $T^*$ is one-to-one, therefore $u_1^*\neq u_2^*$. Let us denote $u_1=(1,0,0,\dots)\in\ell_2$. Then $e_1=T(u_1)\in B$ and $h_i(e_1)=h_i(Tu_1)=u_i^*(u_1)=1.$ Also $h_i(e_1)=1=\sup h_i(B)$ $(i=1,2)$. It follows that 
\begin{equation*}
    1=\sup\{h_i(Tu)\colon u\in B_{\ell_2}\}=\sup\{u_i^*(u)\colon u\in B_{\ell_2}\}=\|u_i^*\|.
\end{equation*}
A contradiction as $\ell_2$ is smooth. Since $f_1(e_1)=1$ and $\sup f_1(D)\leq 1 $ we have that $f_1\in \D_{(X,\seminorm{\cdot})}(e_1)$. Now, let $\{x_n\}_{n}\subset B_{(X,\seminorm{\cdot})}$ be such that $f_1(x_n)\to 1$ as $n\to +\infty$.  For every $n\in\N$, there exist $\lambda_n\in [0,1]$, $y_n \in B$ and $z_n\in B_{(X,\|\cdot\|)}$ such that $x_n=\lambda_n y_n + (1-\lambda_n)z_n$. Observing that $|f_1(z_n)|<\frac{1}{3}$ for every $n$, we get $\lambda_n\to 1$ as $n\to \infty$. Hence, $f_1(y_n)\to 1$ as $n\to \infty$. For $n\in\N$, let $b_n\in B_{\ell_2}$ be such that $T(b_n)=y_n$. By our hypothesis, we have \[f_1(y_n)=f_1(T(b_n))=(T^*f_1)(b_n)\to 1\]
    By the definition of the operator $T$, the element $T^*(f_1)$ can be represented by the norm one element of $\ell_2$, defined by $z=(1,0,0,\dots)$. Observe that $z(b_n)\to 1$, $z\in S_{\ell_2}$, and $\{b_n\}_{n}\subset B_{\ell_2}$. By uniform convexity of $\ell_2$, we have that $b_n\to z$. By continuity of the operator $T$, we get that the sequence $\{y_n\}_{n}$ converges to $e_1$. Therefore, since $\{\lambda_n\}_{n}$ converges to $1$, we get that $\{x_n\}_{n}$ converges to $e_1$, which proves that $e_1$ satisfies property NSE.
\end{proof}

By combining  Theorems~\ref{th: equiv aLUR}~and~\ref{t: nonreflexive}, and Observation~\ref{obs: aLUR implies NSE}, we obtain the following characterization of reflexivity.

\begin{corollary}\label{c: characterizationreflexive}
    A Banach space $X$ is reflexive if and only if for every equivalent norm $\|\cdot\|$ on $X$ the set of all aLUR points of $S_{(X,\|\cdot\|)}$ coincides with the set of all NSE points of $S_{(X,\|\cdot\|)}$.
\end{corollary}
\medskip

\noindent Let us conclude the paper with a couple of remarks.

\begin{remark}\label{r: BaHuLigap}
    In \cite{BaHuLi04}*{Proposition 4.4} and \cite{BaHuLi04}*{Corollary 4.6} the authors claim that, for a point $x\in S_X$, the following two implications hold
    \begin{enumerate}
        \item NSE $\Rightarrow$ aLUR; 
        \item \textit{w}-NSE $\Rightarrow$ \textit{w}-aLUR,
    \end{enumerate}
  where $x\in S_X$ satisfies \textit{w}-NSE if $\{x_n\}_n$ converges weakly to $x$, whenever $x^*\in \D(x)$ and  $\{x_n\}_n\subset S_X$ are such that $x^*(x_n)\to 1$.  In light of Theorem \ref{t: nonreflexive}, both implications are not true in general (notice that NSE implies \textit{w}-NSE and in the proof of Theorem \ref{t: nonreflexive} we prove that the point $x_1$ is not \textit{w}-aLUR). By a careful reading of $(g)\Rightarrow (a)$ in \cite{BaHuLi04}*{Proposition 4.4} we realized that there is no reason for the sequence $\{y^*(y_n)\}_n$ to converge to 1. Moreover, it is worth to mention that the proof of \cite{BaHuLi04}*{Proposition 5.7} is based on \cite{BaHuLi04}*{Proposition 4.4} and  \cite{BaHuLi04}*{Corollary 4.6}, therefore it contains a gap.
\end{remark}

\begin{remark}\label{r: debesomfix}
In light of Observation~\ref{obs: aLUR implies NSE} and Theorem~\ref{th: equiv aLUR}, the results concerning aLUR points contained in \cite{DEBESOMALUR} remain true.
\end{remark}

\subsection*{Acknowledgements} 
We thank Profs. Constantin Z\u alinescu and Tommaso Russo for their valuable comments on the preliminary manuscript, and the anonymous referee for carefully reading it and suggesting fruitful improvements.

\end{document}